\documentclass[12pt]{amsart} 
\usepackage{amssymb,latexsym}
\usepackage{amsmath}
\usepackage{graphicx}
\usepackage{color}
\newtheorem{thm}{Theorem}[section]  
\newtheorem{cor}[thm]{Corollary}
\newtheorem{defin}[thm]{Definition}

\newtheorem{prop}[thm]{Proposition} 
 
\begin{document}  
\newcommand{\aA}{\mbox {\sc a}}
\newcommand{\bB}{\mbox {\sc b}}
\newcommand{\xX}{\mbox {\sc x}}
\newcommand{\yY}{\mbox {\sc y}} 
\newcommand{\aaa}{\mbox{$\alpha$}}
\newcommand{\map}{\mbox{$\rightarrow$}}
\newcommand{\ccc}{\mbox{$\cal C$}}
\newcommand{\kkk}{\mbox{$\kappa$}}
\newcommand{\Aaa}{\mbox{$\mathcal A$}}
\newcommand{\Fff}{\mbox{$\mathcal F$}}  
\newcommand{\bbb}{\mbox{$\beta$}}
\newcommand{\sss}{\mbox{$\sigma$}}  
\newcommand{\ddd}{\mbox{$\delta$}} 
\newcommand{\rrr}{\mbox{$\rho$}} 
\newcommand{\Ggg}{\mbox{$\Gamma$}}
\newcommand{\ttt}{\mbox{$\tau$}} 
\newcommand{\bdd}{\mbox{$\partial$}}
\newcommand{\zzz}{\mbox{$\zeta$}}
\newcommand{\Ss}{\mbox{$\Sigma$}}
\newcommand{\Ddd}{\mbox{$\Delta$}}
\newcommand{\Fa}{\mbox{$\Fff_A$}}
\newcommand{\Fb}{\mbox{$\Fff_B$}}
\newcommand{\Thp}{\mbox{$\Theta_P$}}
\newcommand{\Thq}{\mbox{$\Theta_Q$}}
\newcommand{\aub} {\mbox{$A \cup_{P} B$}}
\newcommand{\awb} {\mbox{$A_- \cup_{P_-} B_-$}}
\newcommand{\xuy} {\mbox{$X \cup_{Q} Y$}}
\newcommand{\xwy} {\mbox{$X_- \cup_{Q_-} Y_-$}}
\newcommand{\avb} {\mbox{$A' \cup_{P'} B'$}}
\newcommand{\xvy} {\mbox{$X' \cup_{Q'} Y'$}}
\newcommand{\px} {\mbox{$P_X$}}
\newcommand{\py} {\mbox{$P_Y$}}
\newcommand{\qa} {\mbox{$Q_A$}}
\newcommand{\qb} {\mbox{$Q_B$}}
\newcommand{\Lll}{\mbox{$\Lambda$}}
\newcommand{\inter}{\mbox{${\rm int}$}}

\title{Multiple genus $2$ Heegaard splittings: a missed case}  
\author{John  Berge and  Martin Scharlemann} 
\date{\today}

\thanks{Second author partially supported by an NSF grant.}

\begin{abstract} A gap in \cite{RS} is explored:  new examples are found of closed orientable $3$-manifolds with possibly multiple genus $2$ Heegaard splittings.  Properties common to all the examples in \cite{RS} are not universally shared by the new examples:  some of the new examples have Hempel distance $3$, and it is not clear that a single stabilization always makes the multiple splittings isotopic.

\end{abstract}

\maketitle

\section{Introduction}

In 1998, Rubinstein and the second author \cite{RS} studied the question of when there could be more than one distinct genus $2$ Heegaard splitting of the same closed orientable $3$-manifold.  The goal of the project was modest:  to provide a complete list of ways in which such multiple splittings could be constructed, but with no claim that each example on the list did in fact have multiple non-isotopic splittings (there could be isotopies from one splitting to another that are not apparent).  Nor was there a claim that the list had no redundancies; a $3$-manifold and its multiple splittings might appear more than once on the list.  Such a list would still be useful, for if every example on the list could be shown to have a certain property, then that property would be true for any closed orientable $3$-manifold $M$ that has multiple genus $2$ splittings.  Two examples were given in \cite{RS}:  
\begin{itemize}
\item  If $M$ is atoroidal then the hyperelliptic involutions determined by the two genus $2$ Heegaard splittings commute.
\item Any two genus $2$ Heegaard splittings of $M$ become isotopic after a single stabilization.
\end{itemize}

Despite this modest goal, the argument in \cite{RS} contains a gap.  In 2008, the first author discovered a class of examples that do not appear on the list and which, moreover, have mathematical properties that distinguish them in important ways from the examples that do appear in \cite{RS}.  It is true that, even for the new examples, the hyperelliptic involutions commute.  But we know of no argument showing that the new examples all share the second property above; that is, we cannot show that the newly discovered multiple splittings necessarily become isotopic after a single stabilization (though they do after two stabilizations).  

A third property, shared by all examples in \cite{RS} but not by some of the new examples, is not listed above because the notion of Hempel distance of Heegaard splittings (see \cite{He}) did not exist at the time \cite{RS} was written.  But a retrospective look (see Section \ref{sect:distance} below) will verify that all the splittings described in \cite{RS} have Hempel distance no greater than $2$, whereas results of the first  author \cite{Be2} illustrate that at least some of the new examples have Hempel distance $3$. (This also verifies that the gap in the argument in \cite{RS} actually led to missed examples.)

\bigskip

The present paper serves as an erratum to \cite{RS}\footnote{The error is on p. 533: The last sentence of the first paragraph of Case 2 should have read, ``The {\em same curves} cannot then be twisted in $X$ since $M$ is hyperbolike." This leaves open an additional possibility for $P_X$, $P_Y$, which appears as Subcase B in Section \ref{sect:gap} below.} and  describes the new examples.   Here is an outline: In Section \ref{sect:dehnderiv} we describe a general method for constructing closed orientable $3$-manifolds that appear to have multiple genus $2$ Heegaard splittings; these examples (called {\em Dehn-derived}) are based around Dehn surgery on a pair of strategically placed curves. It follows from the construction that always the hyperelliptic involutions of the alternate splittings coincide. 

It is not immediately obvious that curves supporting Dehn-derived examples can be found, but in Sections \ref{sect:construction1} and \ref{sect:construction2} we give three specific classes of examples.  The classes are denoted $M_H$ (Section \ref{sect:construction1}), $M_{\times I}$ and $M_{hybrid}$ (Section  \ref{sect:construction2}).  ($M_H$ can be viewed as a third variation of \cite[Example 4.2]{RS}.)   For the examples $M_H$ and $M_{hybrid}$ a single stabilization suffices to make the alternate splittings equivalent, but this property is at least not apparent in most cases of $M_{\times I}$. 

Section \ref{sect:gap} describes how the proof of \cite[Lemma 9.5, Case 2]{RS} went astray and how it needs to be altered to fix the gap.  The upshot is that Dehn-derived examples, as described in Section \ref{sect:dehnderiv}, do fill the gap in the original proof. In Section \ref{sect:taxonomy} it is further shown, using new results in \cite{Be}, that {\em any} Dehn-derived example is in fact of type $M_H$, $M_{\times I}$ or $M_{hybrid}$.  Finally, in Section \ref{sect:distance} we verify that all of the old examples that are listed in \cite{RS} are of Hempel distance $2$, whereas at least some Dehn-derived examples are of distance $3$.  (It is easy to see that all Dehn-derived examples are of distance no more than $3$.)

\section{Dehn derived multiple splittings}  \label{sect:dehnderiv}

A {\em primitive $k$-tuple} of curves in the boundary of a genus $g$ handlebody $H$ is a collection $\lambda_1, ... \lambda_k \subset \bdd H$ of $k \leq g$ disjoint simple closed curves so that, for some properly embedded collection $D_1,...,D_k$ of disks in $H$,  $|\lambda_i \cap D_j| = \delta_{ij}, 1 \leq i, j \leq k$.  It is easy to see that the closed complement in $H$ of such a collection of meridian disks is a genus $g - k$ handlebody.  In particular, if $k = g$ then $\lambda_1, ... \lambda_g$ is called a {\em complete set of primitive curves} and the corresponding collection of disks $D_1,...,D_g$ is called a {\em complete set of meridian disks}.  The closed complement of a complete set of meridian disks in $H$ is a $3$-ball. 

Suppose $\Lambda = \lambda_1, ... \lambda_k \subset \bdd H$ is a primitive $k$-tuple of curves in $H$ and let $\aaa_1, ..., \aaa_k$ be the properly embedded collection of curves in $H$ obtained by pushing $\Lambda$ slightly into the interior of $H$.  We can view $H$ as the boundary connect sum of a genus $g-k$ handlebody $H'$ and $k$ solid tori $W_1, ..., W_k$,  with $\lambda_i$ a longitude of $W_i$ and so $\aaa_i$ a core curve of $W_i$.  Then  Dehn surgery on $\aaa_i \subset W_i$ still gives a solid torus.  Hence any Dehn surgery on the family of curves $\aaa_1, ..., \aaa_k$ leaves $H$ still a handlebody.   

\begin{defin} \label{defin:doublyprim} Suppose  $M_0 = H_a \cup_S H_b$ is a Heegaard splitting of a closed $3$-manifold $M_0$.  A simple closed curve $\lambda \subset S$ is {\em doubly primitive} if $\lambda$ is a primitive curve in both handlebodies $H_a$ and $H_b$.
\end{defin}

Suppose $M_0$ is a closed orientable $3$-manifold and that $M_0 = H_a \cup_S H_b$ is a genus $2$ Heegaard splitting of $M_0$.  Suppose further that $\lambda_1, \lambda_2 \subset S$ are two disjoint doubly primitive curves in $S$.

\begin{prop} \label{prop:dehnderiv} Suppose $M$ is a manifold obtained by some specified Dehn surgeries on $\lambda_1$ and $\lambda_2$.  For $i = 1, 2$, let $A_i$ (resp. $B_i$) be the manifold obtained from the handlebody $H_a$ (resp. $H_b$) by pushing the curve $\lambda_i$ into $\inter(H_a)$ (resp $\inter(H_b)$) and performing the specified Dehn surgery on the curve.  

Then $A_1 \cup_S B_2$ and $ A_2 \cup_S B_1$ are two (possibly different) genus $2$ Heegaard splittings of $M$.  
\end{prop}

\begin{proof}  $A_i$ (resp $B_i$) is obtained from $H_a$ (resp $H_b$) by Dehn surgery on a pushed in copy $\aaa_i$ of a single primitive curve in $S$.  It was just observed that this makes each $A_i$ (resp. $B_i$) a handlebody.
\end{proof}

\begin{defin} \label{defin:dehnderived} 
Two genus $2$ Heegaard splittings $X \cup_Q Y$ and $A \cup_P B$ of a closed $3$-manifold $M$ are called {\rm Dehn derived} (from the splitting $M_0 = H_a \cup_S H_b$ via $\lambda_1 \cup \lambda_2 \subset S$) if the two splittings are created as in Proposition \ref{prop:dehnderiv}.
\end{defin}

\begin{cor} \label{cor:hyper}  Suppose $M = A \cup_P B = X \cup_Q Y$ are a Dehn-derived pair of Heegaard splittings.  Then the two hyperelliptic involutions of $M$, one determined by the Heegaard splitting $A \cup_P B$ and the other by the Heegaard splitting $X \cup_Q Y$, coincide.  
\end{cor}

\begin{proof}   Let $M_0 = H_a \cup_S H_b$ be the Heegaard split $3$-manifold from which the two splittings of $M$ are Dehn derived, via $\lambda_1 \cup \lambda_2 \subset S$.  The hyperelliptic involution preserves the isotopy class (though perhaps reversing the orientation) of any simple closed curve in $S$.  We may then position $\lambda_i$ so that the curves are preserved (reversing orientation) by the hyperelliptic involution on $M_0 = H_a \cup_S H_b$.  Then the hyperelliptic involution on $M_0$ naturally induces a single hyperelliptic involution on $M$. \end{proof}

\section{A simple set of examples} \label{sect:construction1}

It is not immediately obvious how to create examples of a Dehn-derived pair of splittings or, very naively, whether examples even exist.  In this section we present and briefly discuss an important concrete class of examples.

Consider a genus $2$ handlebody $H$, constructed from two $0$-handles by connecting them with three $1$-handles.  With this structure $H$ has a natural $\mathbb{Z}_3$ symmetry, shown as $\frac{2\pi}{3}$ rotation about the green axis in Figure \ref{fig:symm3}.  Let $\lambda_1 \subset \bdd H$ be the red curve shown in the figure and $\lambda_2, \lambda_3$ be the other two simple closed curves to which $\lambda_1$ is carried by the $\mathbb{Z}_3$ symmetry.  Then each $\lambda_i$ is a primitive curve on $\bdd H$ and, indeed, any two of the curves, say $\lambda_1, \lambda_2$ constitute a complete set of primitive curves (that is, a primitive pair).  In this case the corresponding pair of meridian disks are the meridian disks of the two $1$-handles through which $\lambda_3$ passes.  

Let $\overline{H}$ be the genus $3$ handlebody obtained by removing from $H$ a neighborhood of the arc in which the axis of symmetry intersects one of the $0$-handles.  It is easy to see that in $\overline{H}$ the collection $\lambda_1, \lambda_2, \lambda_3 \subset \bdd \overline{H}$ is a complete set of primitive curves, that is a primitive $3$-tuple.  

 \begin{figure}[ht!]
    \centering
    \includegraphics[scale=0.7]{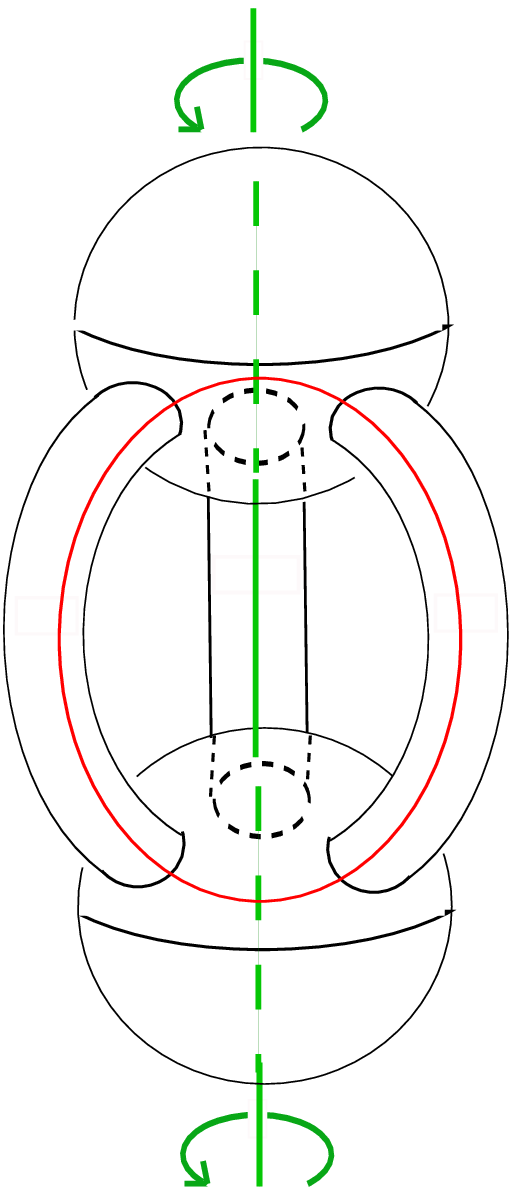}
    \caption{} \label{fig:symm3}
    \end{figure}

To construct some Dehn-derived pairs of Heegaard splittings, begin with two genus $2$ handlebodies $A$ and $B$, on each of whose boundaries lie three disjoint simple closed curves corresponding to $\lambda_1, \lambda_2, \lambda_3 \subset \bdd H$.  Let $\lambda_{ia} \subset \bdd A$ (resp $\lambda_{ib} \subset \bdd B$) be the curve corresponding to $\lambda_i$ in $A$ (resp $B$), for each $1 \leq i \leq 3$.  Adopting (for comparison purposes) notation from \cite[Section 4.2]{RS}, let $\aaa_{na}, \aaa_{sa}, \rrr_a \subset A$ be the triple of curves obtained by pushing $\lambda_{1a}, \lambda_{2a}, \lambda_{3a}$ into the interior of $A$ and let $\aaa_{nb}, \aaa_{sb}, \rrr_b \subset B$ be the triple of curves obtained by pushing $\lambda_{1b}, \lambda_{2b}, \lambda_{3b}$ into the interior of $B$.  Let $N$ be a manifold constructed by identifying an annular neighborhood of $\lambda_{1a}$ in $\bdd A$ with an annular neighborhood of $\lambda_{1b}$ in $\bdd B$ and an annular neighborhood of $\lambda_{2a}$ in $\bdd A$ with an annular neighborhood of $\lambda_{2b}$ in $\bdd B$.  (After the identification, call the annuli $\Aaa_n$ and $\Aaa_s$ with core curves $\lambda_1, \lambda_2$ respectively .)  Then identify the two $4$-punctured spheres $\bdd A - (\Aaa_n \cup \Aaa_s)$ and $\bdd B - (\Aaa_n \cup \Aaa_s)$ by any homeomorphism.  

This construction defines a genus $2$ Heegaard structure on $N$, of course, but it also defines a genus $2$ Heegaard splitting on $M_0$, the manifold obtained from $N$ by arbitrary Dehn surgery on just the two curves $\rrr_a \subset A$ and $\rrr_b \subset B$, for surgery on these pushed-in primitive curves leaves $A$ and $B$ still handlebodies, handlebodies which we denote respectively $H_a$ and $H_b$.  What's more, the curves $\lambda_1, \lambda_2$ are each primitive in both $H_a$ and $H_b$ (though they are not necessarily a primitive pair in either).  Thus the Heegaard splitting $M_0 =  H_a \cup H_b$ gives rise to  two potentially different genus $2$ Heegaard structures on any manifold $M_H$ that is obtained by simultaneously doing further Dehn surgery on the two curves $\lambda_1, \lambda_2$.  That is, a manifold $M_H$ obtained by arbitrary Dehn surgery on all four curves $ \lambda_1, \lambda_2, \rrr_a, \rrr_b \subset N$ has two possibly distinct genus $2$ Heegaard splittings, Dehn derived from the Heegaard splitting $M_0 = H_a \cup H_b$.  One Heegaard structure $M_H = A_1 \cup B_2$ is obtained by pushing $\lambda_1$ to $\aaa_{na} \subset \inter A$ and $\lambda_2$ to $\aaa_{sb} \subset \inter B$ before doing Dehn surgeries on the four curves; the other $M_H = A_2 \cup B_1$ is obtained by pushing $\lambda_1$ to $\aaa_{nb} \subset \inter B$ and $\lambda_2$ to $\aaa_{sa} \subset \inter A$ before doing the Dehn surgeries.  In each case, exactly two of the four Dehn surgered curves lie in each handlebody $A$ and $B$ before the Dehn surgery, and in that handlebody are a pushed-in primitive pair.   

\begin{prop}  \label{prop:stabil}  The two Heegaard splittings $A_1 \cup B_2$ and $A_2 \cup B_1$ of $M_H$ become isotopic after at most a single stabilization.
\end{prop}

\begin{proof}  Let $\overline{A}$ and $\overline{B}$ be the genus $3$ handlebodies derived from $A$ and $B$ respectively, just as $\overline{H}$ was derived from $H$.   Here is a natural genus $3$ Heegaard splitting of $M_H$:  in contrast to the construction above, push {\em both} $\lambda_1$ to $\aaa_{na}$ and $\lambda_2$ to $\aaa_{sa}$, so both curves (as well as $\rrr_a$)  lie in $A$ before doing the Dehn surgeries.  Although $A$ may no longer be a handlebody after the Dehn surgeries, it follows from the discussion above that the result on $\overline{A}$ of the surgery on the three curves $\aaa_{na}, \aaa_{sa}, \rrr_a \subset \overline{A} \subset A$ is still a genus three handlebody $\overline{A}'$.  The complement of $\overline{A}'$ in $M_H$ is also a handlebody $B'$:  a single $1$-handle is added to $B$ and surgery is done on the single curve $\rrr_b \subset B$.  Thus $M_H = \overline{A}' \cup B'$ is a genus $3$ Heegaard splitting of $M_H$.

It's fairly easy to see that this Heegaard splitting is a stabilization of $A_1 \cup B_2$ (and so, symmetrically, $A_2 \cup B_1$).  Indeed, an alternate way to construct $\overline{A}' \cup B'$ is to begin with $A_1 \cup B_2$ and add to $A_1$ (and so subtract from $B_2$) a regular neighborhood of the curve $\aaa_{sb} \subset \inter(B)$ and a straight arc from $\bdd B$ to $\aaa_{sb}$.  From this point of view, the inclusion $B' \subset B_2$ defines a genus $3$ Heegaard splitting of the genus $2$ handlebody $B_2$, and any such Heegaard splitting is necessarily stabilized (see \cite[Lemma 2.7]{ST}).  The pair of stabilizing disks are also a pair of stabilizing disks for $\overline{A}' \cup B'$
\end{proof}

\section{A second construction, and a hybrid} \label{sect:construction2}

Here is another natural, but less naive, way to find disjoint pairs of primitive curves on the boundary of a genus $2$-handlebody and so to create a Dehn-derived pair of Heegaard splittings.  Let $F$ denote a torus with the interior of a disk removed.  Then $F \times I$ is a genus $2$-handlebody.  For $\gamma$ any properly embedded essential simple closed curve in $F$, $\gamma \times \{ 0 \}$ (or symmetrically $\gamma \times \{ 1 \}$) is a primitive curve in the handlebody $F \times I$.  Indeed, for $\ddd$ a properly embedded arc in $F$ intersecting $\gamma$ once, $\ddd \times I$ is a meridian disk in $F \times I$ that intersects $\gamma \times \{ 0 \}$ exactly once.

Following this observation, and the example of the previous section, here is a recipe for constructing candidate $3$-manifolds.  Begin with two copies $A$ and $B$ of the surface $F$ and choose two essential (not necessarily disjoint) simple closed curves $\aaa_0, \aaa_1 \subset A$ and two essential (not necessarily disjoint) simple closed curves $\bbb_0, \bbb_1 \subset B$.  Let $\lambda_{0a} = \aaa_0 \times \{ 0 \} \subset \bdd (A \times I),   \lambda_{1a} = \aaa_1 \times \{ 1 \} \subset \bdd (A \times I), \lambda_{0b} = \bbb_0 \times \{ 0 \} \subset \bdd (B \times I), \lambda_{1b} = \bbb_1 \times \{ 1 \} \subset \bdd (B \times I)$.   Identify an annular neighborhood of $\lambda_{0a}$ in $A \times \{ 0 \}$ with an annular neighborhood of $\lambda_{0b}$ in $B \times \{ 0 \}$ and call the core curve of the resulting annulus $\lambda_0$.  Similarly identify an annular neighborhood of $\lambda_{1a}$ in $A \times \{ 1 \}$ with an annular neighborhood of $\lambda_{1b}$ in $B \times \{ 1 \}$ and call the core curve of the resulting annulus $\lambda_1$.  Complete the identification of $\bdd(A \times I)$ with $\bdd(B \times I)$ along the remaining $4$-punctured sphere arbitrarily.  Call the resulting closed $3$-manifold $M_0$, with Heegaard splitting $M_0 = (A \times I) \cup (B \times I)$.  

The $3$-manifold $M_{\times I}$ obtained from $M_0$ by doing arbitrary Dehn surgeries to the simple closed curves $\lambda_0$ and $\lambda_1$ has a Dehn-derived pair of Heegaard splittings: one comes from first pushing $\lambda_0$ into $A \times I$ and $\lambda_1$ into $B \times I$ before the Dehn surgery, the other comes from first pushing $\lambda_1$ into $A \times I$ and $\lambda_0$ into $B \times I$ before the Dehn surgery. 

\bigskip

\noindent {\bf Remarks on stabilization} It is not apparent to us that a single stabilization will make the two Dehn-derived splittings of $M_{\times I}$ equivalent.  The argument of Proposition \ref{prop:stabil}   does not immediately carry over: if both curves $\lambda_{0a} $ and $ \lambda_{1a}$ are pushed into $A \times I$ there is no apparent arc so that the complement $\overline{A \times I}$ of a neighborhood of the arc in $A \times I$  is a genus $3$ handlebody after an arbitrary Dehn surgery on the pushed in  $\lambda_{0a} $ and $ \lambda_{1a}$.  If there is a proper arc $\gamma$ in $A$ that intersects both curves $\aaa_0 \subset A$ and $\aaa_1 \subset A$ in a single point, then the complement   $\overline{A \times I}$ after pushing the interior of $\gamma$ into $A \times I$ is a genus $3$ handlebody, and so a single stabilization suffices, but having such an arc $\gamma$ is not the general situation.  (What is required for such an arc $\gamma$ to exist is that the slopes of $\aaa_0$ $\aaa_1$ in $A$ are a distance at most two apart in the Farey graph \cite[Figure 1]{Mi}.  In that case $\gamma$ has the slope that is incident to the slopes of both $\aaa_0$ and $\aaa_1$ in the Farey graph.)  

On the other hand, it is relatively easy to show that two stabilizations suffice to make the two splittings equivalent.  To see this, push both $\lambda_{0} $ and $ \lambda_{1}$ into $A \times I$ and connect them to respectively $A \times \{ 0 \}$ and $A \times \{ 1 \}$ by straight arcs.  Then add a regular neighborhood of the arcs and of the pushed in curves $\lambda_{0} $ and $ \lambda_{1}$ to $B \times I$ to create a genus $4$ handlebody $\overline{B \times I}$ and simultaneously subtract the regular neighborhood from $A \times I$ to get the genus $4$ handlebody $\overline{A \times I}$.  The resulting genus $4$ Heegaard splitting $M_0 = \overline{A \times I} \cup \overline{B \times I}$ becomes a Heegaard splitting $H_a^+ \cup H_b^+$ of $M_{\times I}$ after the prescribed Dehn surgery on  $\lambda_{0} $ and $ \lambda_{1}$.  Using the argument of Proposition \ref{prop:stabil} it is easy to see that the Heegaard splitting $H_a^+ \cup H_b^+$ destablizes to the genus $3$ splitting obtained by instead pushing $\lambda_0$ into $B \times I$ and then adding to $B \times I$ a regular neighborhood of $\lambda_1 \subset (A \times I)$ and a straight arc attaching it to $A \times \{ 1 \}$.  The argument of Proposition \ref{prop:stabil} applied again shows that this Heegaard splitting destabilizes to the genus $2$ splitting in which $\lambda_0$ is pushed into $B \times I$ and $\lambda_1$ into $A \times I$, one of the Dehn-derived splittings.  But this destabilization process is clearly symmetric: we could equally well have destabilized to the other genus $2$ splitting, in which $\lambda_0$ is pushed into $A \times I$ and $\lambda_1$ into $B \times I$, and this is the other Dehn-derived splitting. 

\bigskip

A further, call it a {\em hybrid} example of a Dehn-derived pair of splittings comes by combining the two constructions above:  Identify annular neighborhoods of $\lambda_1, \lambda_2 \subset \bdd H$ from Section \ref{sect:construction1} with annular neighborhoods of $\lambda_{0b}, \lambda_{1b} \subset \bdd (B \times I)$ and identify the rest of $\bdd H$ with the rest of $\bdd (B \times I)$ in any way.  This gives a closed $3$-manifold $N$ with a Heegaard splitting $H \cup (B \times I)$.  Let $M_0$ be a $3$-manifold obtained by doing an arbitrary Dehn surgery on $\lambda_3 \subset \bdd H$, after pushing it into $\inter(H)$.  Then $M_0$ has the genus $2$ Heegaard splitting (exploiting the notation used above) $M_0 = H_a \cup (B \times I)$.  Let $M_{hybrid}$ be a closed $3$-manifold obtained from $M_0$ by arbitrary Dehn surgeries on the two remaining curves $\lambda_1, \lambda_2 \subset \bdd H_a \subset M_0$.  The Dehn-derived pair of Heegaard splittings for $M_{hybrid}$ is obtained by alternatively pushing $\lambda_1$ into $H_a$ and $\lambda_2$ into $B \times I$ or vice versa.  A single stabilization suffices to make the two splittings equivalent, essentially by the same argument as for $M_H$, in Proposition \ref{prop:stabil}.  

\section{Filling the gap in \cite{RS}} \label{sect:gap}

The gap in \cite{RS} arises because of a faulty sentence in the midst of a long and technical argument which would be difficult to summarize.  We see no good alternative to simply jumping into that argument at a reasonable breaking point and inserting the argument we now believe to be complete.  The jumping in point is on p. 533, in the midst of trying to prove that all cases of multiple genus $2$ Heegaard splittings have been covered in the earlier examples listed in that paper.  Here $M$ is a closed hyperbolike $3$-manifold with Heegaard splittings $M = \aub = \xuy$.  The two splitting surfaces $P$ and $Q$ have been made to coincide on sub-surfaces $P_0 \subset P$ and $Q_0 \subset Q$.  Then $P - P_0$ consists of annular components $P_X$ and $P_Y$ properly embedded in the handlebodies $X$ and $Y$ respectively, and $Q - Q_0$ consists of annular components $Q_A, Q_B$ properly embedded in the handlebodies $A, B$ respectively.  With this as background, we now re-enter the proof of \cite[Theorem 9.4, Case 2]{RS} on page 533, at first echoing what is written there as the proof of Subcase A below.  In filling the gap in the argument we also broaden the possible outcome, as expressed in Proposition \ref{prop:corrected}:

\bigskip

{\bf Case 2:}  $P_X$ and $P_Y$ are both non-empty and the end of each
curve in $\bdd P_X \cup \bdd P_Y$ is parallel to one of $c_1$ or $c_2$.

\begin{prop} \label{prop:corrected}  In this case, either 
\begin{enumerate}
\item[I)] the splittings $M = \aub = \xuy$ are related as in \cite[Example 4.4]{RS} or

\item[II)]the two splittings $M = \aub = \xuy$ are Dehn-derived from a single genus $2$ Heegaard splitting of another manifold $M_0.$
\end{enumerate}
\end{prop}

\begin{proof}  If at least one annulus in each of $P_X$ or $P_Y$ is non-separating,
then together they would give a non-separating, hence essential, torus
in $M$. This contradicts our assumption that $M$ is hyperbolike. So we
may as well assume that each annulus in $P_Y$ is separating.  Hence the
ends of $P_Y$ are twisted in $Y$ (see \cite[Definition 5.4]{RS}).  No end of $P_Y$ can also be twisted in $X$, for the union along the curve of the solid tori (one in $X$, one in $Y$) on which the curve is a torus knot would be a Seifert submanifold of $M$, contradicting the assumption that $M$ is hyperbolike.  

\bigskip

{\bf Subcase A:}  Some end of $P_Y$ is parallel to an end of $P_X$. [The gap in \cite{RS} was to view this as the only possibility.]

\bigskip

In this case, by \cite[Lemma 5.6]{RS} all of $P_X$ is a collection of parallel non-separating longitudinal annuli in $X$.   If $P_Y$ has ends at both $c_1$ and $c_2$ then neither curve can be twisted in $X$.  In this case each annulus in $P_X$ is non-separating and so has ends that are non-parallel in $Q$.  This implies that each annulus in $P_X$ has one end at $c_1$ and one at $c_2$.  Attach such an annulus in $X$ to the tori in $Y$ on which the $c_i$ are
twisted.  The boundary of the thickened result would exhibit a Seifert
manifold in $M$, again contradicting the assumption that $M$ is
hyperbolike.  We conclude that  $P_Y$ has ends only at $c_2$, say. 

If there were three or more annuli in $P_Y$ (hence six or more ends of $\bdd P_Y$ at $c_2$) then there would be at least four ends of $P_X$ at $c_2$.  No annulus in $P_X$ could have both ends at $c_2$ (since $c_2$ is not twisted in $X$) so there would also be at least four ends of $P_X$ at $c_1$.  This would contradict \cite[Lemma 9.5]{RS}.  So we conclude that $P_Y$ is made up of one or two annuli.    If it's two annuli, necessarily separating and parallel in $Y$, then, again by \cite[Lemma 9.5]{RS} some annulus in $P_X$ has an end at $c_2$.  It cannot have both ends at $c_2$ and must be non-separating and longitudinal in $X$, since $c_2$ is not twisted in $X$.  In this case the relation
between $P$ and $Q$ can be seen as follows (See \cite[Figure 32]{RS}): In \cite[Example 4.4, Variation 2]{RS}, let $P$ be the  splitting given there with Dehn surgery
curve in $\mu_{a_+}$ and $Q$ be the same splitting given there but with
Dehn surgery curve in $\mu_{a_-}$. To view these simultaneously as
splittings of the same manifold $M$, of course, the Dehn surgery curve
has to be moved from $\mu_{a_+}$ to $\mu_{a_-}$, dragging some annuli
along, until the splitting surfaces $P$ and $Q$ intersect as described.

Suppose then that $P_Y$ is a single annulus.  It may have both ends on $P_0$ or it may have one
end on $P_0$ and one end on an end of $P_X$.  (If both ends of $P_Y$ were also ends of $P_X$ then these, together with ends of $P_X$ at $\bdd P_0$ parallel to $c_2$ would exhibit more than two annuli in $P_X$, hence more than two ends of $P_X$ at $c_1$, contradicting \cite[Lemma 9.5]{RS}.)  If $P_Y$ has  one end on $P_0$ and one end on an end of $P_X$, the initial splitting by $Q$ is as in \cite[Example 4.4, Variation 1]{RS} ($X = A_- \cup \sss$), with a Dehn surgery curve lying in
$\mu_{b_+}$, say.  If the splitting is altered by first putting the
Dehn surgery curve in $\mu_{a_+}$  (yielding the
same manifold $M$), then altering as in \cite[Example 4.4]{RS} (i. e. considering \aub\ where $B = B_- \cup \sss$) and
then dragging the Dehn surgery curve from  $\mu_{a_+}$ to
$\mu_{b_+}$, pushing before it an annulus from the $4$-punctured sphere
along which $A_-$ and $B_-$ are identified, we get the splitting
surface $P$, intersecting $Q$ as required. (See \cite[Figure 33]{RS}.)  This completes the proof that I) holds in Subcase A.

\bigskip

{\bf Subcase B:}  No end of $P_Y$ is parallel to an end of $P_X$. 

\bigskip

In view of \cite[Lemma 9.5]{RS}, in this subcase $P_Y$ and $P_X$ each consist of exactly one separating annulus, $P_Y$ twisted in $Y$ with boundary curves parallel to $c_2$ (say) and $P_X$ twisted in $X$ with boundary curves parallel to $c_1$.  This case is symmetric: the annulus in $Q$ lying between the ends of $P_Y$ is $Q_A$ (say) and the annulus in $Q$ lying between the ends of $P_X$ is exactly $Q_B$.  The annulus $P_Y$ cannot be parallel to the annulus $Q_A$ (else $P_0$ and $Q_0$ could be extended to include both) but rather the region between them is a solid torus $W_2 = A \cap Y$ on whose boundary the cores of the annuli are torus knots.  Similarly $B \cap X$ is a solid torus $W_1$ on whose boundary the cores of the annuli $Q_B$ and $P_X$ are torus knots.  The annulus $P_Y$  $\bdd$-compresses in $Y$ to become a separating disk; it follows that $Y - W_2 = B \cap Y$ is a genus $2$ handlebody $H_{BY}$ on which the core $a_2$ of the annulus $P_Y$ is primitive.   Symmetrically, the curve $c_1$ (viewed as the core of the annulus $Q_B$) is primitive in $H_{BY}$, the curve $c_2$ is primitive in the genus $2$ handlebody $H_{AX} = X - W_1 = X \cap A$, as is the core curve $a_1$ of $P_X$.  

Here is another way to describe the manifold $M$ above:  begin with the two genus $2$ handlebodies $H_{AX}$ (which contains disjoint primitive simple closed curves $a_1, c_2$ on its boundary) and $H_{BY}$ (which contains disjoint primitive simple closed curves $a_2, c_1$ on its boundary).  Construct a closed $3$-manifold $M_0$ by identifying $\bdd H_{AX}$ to $\bdd H_{BY}$ by a homeomorphism that identifies $a_i$ with $c_i$, $i = 1, 2$.  Call the resulting curves $\aaa_1, \aaa_2$.  Now recover $M$ from $M_0$ by removing a tubular neighborhood of each $\aaa_i$ and replacing with the solid torus $W_i$; equivalently, do an appropriate surgery on each $\aaa_i$ in $M_0$.  The two Heegaard splittings of $M$ are then seen as follows:  if $\aaa_1$ is pushed into $H_{AX}$ and $\aaa_2$ into $H_{BY}$ before the surgery on the curves, then the resulting Heegaard splitting is $M = X \cup_Q Y$; if $\aaa_1$ is pushed into $H_{BY}$ and $\aaa_2$ into $H_{AX}$ before the surgery then the resulting splitting is $M = A \cup_P B$. That is, the splittings of $M$ are both Dehn-derived from the Heegaard splitting $M_0 = H_{AX} \cup H_{BY}$.  Thus II) holds in Subcase B.
 \end{proof}

\section{A taxonomy of Dehn-derived splittings} \label{sect:taxonomy}

Sections \ref{sect:construction1} and \ref{sect:construction2} give concrete examples of pairs of Dehn-derived fillings.  In this section we show that these examples in fact constitute all pairs of Dehn-derived splittings.   The argument exploits Berge's classification of pairs of primitive curves on genus $2$ handlebodies \cite{Be}, though the classification here is slightly different. 

Let $H$ be a genus $2$ handlebody, with $\lambda_1, \lambda_2, \lambda_3 \subset \bdd H$ the disjoint simple closed curves described in Section \ref{sect:construction1}.  Denote by $\rrr$ the curve in the interior of $H$ obtained by pushing $\lambda_3$ into $H$ and let $H_{surg}$ denote the handlebody obtained from $H$ by a specified Dehn surgery on $\rrr \subset \inter(H)$.   As in Section \ref{sect:construction2}, let $F$ denote a torus with the interior of a disk removed.

\begin{prop}[Berge]  \label{prop:taxonomy} Suppose $\aaa$ and $\bbb$ are disjoint non-parallel primitive curves on the boundary of a genus $2$ handlebody $H$.  Then either 

\begin{enumerate}

\item[A)] there is a Dehn surgery on $\rrr \subset H$ and a homeomorphism $h: H \to H_{surg}$ so that $h(\aaa) = \lambda_1 \subset \bdd H_{surg} $ and $h(\bbb) = \lambda_2 \subset \bdd H_{surg} $ or

\item[B)] there is a homeomorphism $h: H \to F \times I$ so that $h(\aaa) \subset F \times \{ 0 \}$ and $h(\bbb) \subset F \times \{ 1\} $.
\end{enumerate}

\end{prop}

\begin{proof}  This classification is a variant of that described in \cite{Be}.  The Type II pair there, as well as some pairs of Type I, are exactly as described in alternative B).  The interest is in the third example of a Type I pair, in \cite[Lemma 3.8 (3) via Figure 3]{Be}.    In that example, $H$ is viewed as divided into two solid tori by a separating disk $D$; let $\lambda_a$ and $\lambda_b$ be longitudes of the two solid tori into which $D$ divides $H$.  Then $\bbb$ is parallel to $\lambda_b$, and $\aaa$ is the band sum, via a band that crosses $D$ once, of $\lambda_b$ with a torus knot on the solid torus containing $\lambda_a$.  This picture is equivalent to letting $\aaa$ be the band sum $\lambda_a \# \lambda_b$ (through $D$) of $\lambda_b$ with $\lambda_a$, and then performing a Dehn surgery on a disjoint copy of $\lambda_a$ that has been pushed into $H$, to become a core of the solid torus on which $\lambda_a$ lies.  Now translate:  relabel $\lambda_b \subset \bdd H$ as $\lambda_2$ and $\lambda_a \subset \bdd H$ as $\lambda_3$.  Then $\lambda_a \# \lambda_b$ corresponds to $\lambda_1$.  The construction just described is then to push $\lambda_3$ into the interior of $H$ and perform some surgery to get $H_{surg}$.  Afterwards $\aaa$ corresponds to $\lambda_1  \subset \bdd H_{surg}$ and $\bbb$ corresponds to $\lambda_2 \subset \bdd H_{surg}$.  This is exactly alternative A).  \end{proof}

Following Propositions  \ref{prop:corrected} and \ref{prop:taxonomy} there is a fairly clear description of the cases of multiple Heegaard splittings that are missing from \cite{RS}.   According to Proposition \ref{prop:corrected} the only missing cases are pairs of splittings that are Dehn-derived from an initial splitting $H_{AX} \cup H_{BY}$ of a manifold $M_0$.  First determine which of alternatives A) and B) apply to the pairs of surgery curves as they lie on the boundaries of the respective handlebodies:  $\{a_1, c_2\} \subset H_{AX}$ or $\{a_2, c_1\} \subset H_{BY}$.  If both are of type A) then the pair of splittings is Dehn-derived as in the construction of $M_H$ in Section \ref{sect:construction1}.  If both are of type B) then the pair of splittings is Dehn-derived as in the construction of $M_{\times I}$ in Section \ref{sect:construction2}.  If one is of type A) and one of type B) then the pair of splittings is Dehn-derived as in the construction of $M_{hybrid}$ in Section \ref{sect:construction2}. 

It is worth mentioning that there is another view of a pair of primitive curves lying on a handlebody as in A) of Proposition \ref{prop:taxonomy}, a view that more closely resembles that in B):  Let $\aaa, \bbb, \gamma$ be simple closed curves in $F$ so that each pair of curves intersects in exactly one point.  (For example, choose curves in $F$ of slopes $0, 1, \infty$.) Then it is fairly easy to see that the three curves $\aaa \times \{ 0 \}, \bbb \times \{ 1 \}, \gamma \times \{ \frac12 \}$ lie in the handlebody $F \times I$ just as $\lambda_1, \lambda_2, \rrr$ lie in $H$ in the description preceding Proposition \ref{prop:taxonomy}.  So the primitive curves in description A) can be made to look like a special case of those in description B), but with the cost that an extra Dehn surgery has to be performed on a specific curve in the interior of $F \times I$.   This is the twisted product view of \cite[3.2]{Be}.  

\section{Distance} \label{sect:distance}

It would seem possible that the Dehn-derived pairs of Heegaard splittings exhibited above could coincidentally all be contained among the examples already listed in \cite{RS}, for there is no claim that the types of examples of multiple Heegaard splittings we have offered here and in \cite{RS} do not overlap.  But in fact there is an invariant which does show that at least some Dehn-derived pairs of Heegaard splittings described above did not already occur in a different guise in \cite{RS}.  This invariant had not yet been introduced when \cite{RS} was written and is called the {\em (Hempel) distance} of the Heegaard splitting \cite{He}.  
We briefly review:

\begin{defin} \label{defin:distance}  A Heegaard splitting $H_1 \cup_S H_2$ has Hempel distance at most $n$ if there is a sequence $c_0, ..., c_n$ of essential simple closed curves in the splitting surface $S$ so that
\begin{itemize}
\item for each $i = 1,..., n$, $c_i \cap c_{i-1} = \emptyset$
\item $c_0$ bounds a disk in $H_1$
\item  $c_n$ bounds a disk in $H_2$
\end{itemize}

If the splitting has distance $\leq n$ but not $\leq n-1$, then the distance $d(S) = n$.
\end{defin}

A Heegaard splitting of distance $0$ is called {\em reducible}; one of distance $\leq 1$ is called {\em weakly reducible}.  Any Heegaard splitting of a reducible manifold is reducible.  A Heegaard splitting of distance $\leq 2$ is said to have the {\em disjoint curve property} \cite{Th}; any Heegaard splitting of a toroidal $3$-manifold has the disjoint curve property (\cite{He}, \cite{Th}).  A weakly reducible genus $2$ Heegaard splitting is also reducible, so an irreducible Heegaard splitting of genus $2$ has distance at least $2$ (\cite{Th}).  

In the other direction we have:

\begin{prop}  Suppose the manifold $M$ has a Dehn-derived pair of Heegaard splittings.  Then each of these Heegaard splittings has Hempel distance at most $3$.
\end{prop}

\begin{proof}  Suppose the splittings are Dehn-derived from a splitting $M_0 = H_a \cup_S H_b$ via the disjoint pair of simple closed curves $\lambda_1, \lambda_2 \subset S$.  With no loss of generality, consider the splitting $M = A \cup_S B$ obtained by pushing $\lambda_1$ into $\inter(H_a)$ and $\lambda_2$ into $\inter(H_b)$ before doing Dehn surgery on the $\lambda_i$.  Since $\lambda_1$ is primitive in $H_a$ there is a properly embedded essential disk $D_a \subset H_a$ that is disjoint from $\lambda_1$.  (For example $D_a$ can be obtained from a meridian disk $D_1 \subset H_a$ that intersects $\lambda_1$ in a single point by band-summing together two copies of $D_1$ along a subarc of $\lambda_1 - D_1$.)  $D_a$ is then also disjoint from the curve $\aaa_1 \subset H_a$ obtained by pushing $\lambda_1$ into $\inter(H_a)$, so $D_a$ remains intact as a meridian of $A$ after surgery on $\aaa_1$.  Hence $\bdd D_a$ and $\lambda_1$ are disjoint curves in $\bdd A$.

Symmetrically, there is a meridian $D_b \subset B$ so that $\bdd D_b$ and $\lambda_2$ are disjoint curves in $\bdd B$.  Then the sequence $\bdd D_a, \lambda_1, \lambda_2, \bdd D_b$ shows that the splitting $A \cup_S B$ has distance at most $3$. \end{proof}

\begin{prop} All examples of multiple Heegaard splittings appearing in \cite[Section 4]{RS} have Hempel distance $\leq 2$.
\end{prop}

\begin{proof}  Following the comments above we can restrict attention to irreducible, atoroidal manifolds.  We briefly run through the examples as they appear in \cite[Section 4]{RS}.  Typically the description of an example $H_1 \cup_S H_2$ in \cite{RS} consists of two parts: A collection of annuli $\mathcal{A} \subset S$ along which  $\bdd H_1$ and $\bdd H_2$ are identified, followed by an arbitrary identification of $\bdd H_1 - \mathcal{A}$ with  $\bdd H_2 - \mathcal{A}$.  From this point of view the simple closed curves $\bdd \mathcal{A} \subset S$ that separate one sort of region from the other will be called the {\em seams} of the Heegaard splitting.  We will observe that in \cite{RS} some seam is always disjoint from an essential disk in $H_1$ and an essential disk in $H_2$.  This demonstrates that the splitting has the disjoint curve property and so  has distance $\leq 2$.

To be specific:  In \cite[Subsection 4.1]{RS}, \cite[Subsection 4.2, Variation 1]{RS} and \cite[Subsection 4.4, Variations 1 and 2]{RS}, the meridians of the $1$-handles $e_a$ and $e_b$ are disjoint from the seams.  \cite[Subsection 4.2, Variation 2]{RS} is slightly more complicated.  It is a bit like the construction in Section \ref{sect:construction1} above:  Handlebodies $A$ and $B$ are identified along neighborhoods of all three curves $\lambda_i, i = 1, 2, 3$, Dehn surgery is done to all three, with $\lambda_1, \lambda_2$ pushed into $A$ and $\lambda_3$ into $B$ (then vice versa).  But there is a meridian of $A$ disjoint from $\lambda_1$ and $\lambda_2$ and a meridian of $B$ disjoint from $\lambda_1$ and $\lambda_3$, so a seam parallel to $\lambda_1$ demonstrates that the splitting of  \cite[Subsection 4.2, Variation 2]{RS} has the disjoint curve property.  

The manifolds in  \cite[Subsection 4.3]{RS} and \cite[Subsection 4.4, Variations 3, 4, and 7]{RS} are all toroidal, so they are of distance $\leq 2$.  What remains are \cite[Subsection 4.4, Variations 5 and 6]{RS} and we adopt the terminology there.  In Variation 5, with, say, $\rrr_a \subset A_-$, the seams that are the boundary of the $4$-punctured sphere $\bdd A_- \cap \bdd \Ggg$ are all disjoint from the meridian of the $1$-handle $e_b \subset B$ and, in $A_-$, any one of these seams together with $\rrr_a$ lie in $A_-$ as two of the $\lambda_i$'s of Section \ref{sect:construction1} above lie in $H$.  In particular, there is a meridian of $A_-$ disjoint from both the seam and from $\rrr_a$.  Thus that seam again illustrates that the splitting has the disjoint curve property.  

The argument for Variation 6 is much the same.  First note that if, in that Variation, Dehn surgeries are done on two curves parallel to $\sss$, then the resulting manifold has a Seifert piece and so has distance $\leq 2$.  So the only change we need to consider from Variation 5 is Dehn surgery on a single curve parallel to $\sss$.  If that curve lies in $B$ the argument for Variation 5 suffices; if it is in $A_-$ this merely forces us to pick a specific seam in the argument for Variation 5, a seam parallel to the new surgery curve.
\end{proof}

In contrast, some of the examples constructed in this paper can be shown to have distance $3$, so they cannot have appeared in any case considered in \cite{RS}.  See \cite{Be2} (also \cite{Sc}) for details.

\vspace{10mm}

\baselineskip 14pt \noindent {\sf jberge@charter.net}

\vspace{10mm}

\baselineskip 14pt \noindent {\sf Department of Mathematics\\ University 
of 
California\\ Santa Barbara, CA 93106\\ mgscharl@math.ucsb.edu}

\end{document}